\documentclass{article}
\usepackage[utf8]{inputenc}
\usepackage{comment}
\usepackage{amssymb,amsmath,amsthm,graphicx,xcolor,mathtools,enumerate,mathrsfs}
\usepackage{fullpage,enumitem}
\usepackage{thmtools}
\usepackage{thm-restate}
\usepackage{hyperref}

\usepackage{hyperref}
\hypersetup{colorlinks, linkcolor={black}, citecolor={black}, urlcolor={blue!50!black}}

\usepackage{cleveref}
\theoremstyle{plain}
\newtheorem{theorem}{Theorem}[section]
\crefname{theorem}{Theorem}{Theorems}

\newtheorem{proposition}[theorem]{Proposition}
\crefname{proposition}{Proposition}{Propositions}

\newtheorem{corollary}[theorem]{Corollary}
\crefname{corollary}{Corollary}{Corollaries}

\newtheorem{lemma}[theorem]{Lemma}
\crefname{lemma}{Lemma}{Lemmas}

\newtheorem{conjecture}[theorem]{Conjecture}
\crefname{conjecture}{Conjecture}{Conjectures}

\crefname{problem}{Problem}{Problem}

\newtheorem{claim}[theorem]{Claim}
\crefname{claim}{Claim}{Claims}

\crefname{observation}{Observation}{Observations}

\crefname{setup}{Setup}{Setups}

\crefname{fact}{Fact}{Facts}

\crefname{algorithm}{Algorithm}{Algorithms}

\crefname{remark}{Remark}{Remarks}

\crefname{example}{Example}{Examples}

\theoremstyle{definition}
\newtheorem{definition}[theorem]{Definition}
\crefname{definition}{Definition}{Definitions}

\crefname{construction}{Construction}{Constructions}

\newtheorem{question}[theorem]{Question}
\crefname{question}{Question}{Questions}

\numberwithin{equation}{section}

\definecolor{DarkDesaturatedBlue}{HTML}{3A3556}
\definecolor{VividOrange}{HTML}{F15918}
\definecolor{PureOrange}{HTML}{FFBA00}
\definecolor{LightGrayishPink}{HTML}{EEC5D5}
\definecolor{VerySoftBlue}{HTML}{B5AFDB}

\usepackage{tikz}
\usetikzlibrary{math}
\usetikzlibrary{snakes,calc,decorations.pathmorphing,through}
\tikzset{snake it/.style={decorate, decoration=snake}}

\definecolor{DarkDesaturatedBlue}{HTML}{3A3556}
\definecolor{VividOrange}{HTML}{F15918}
\definecolor{PureOrange}{HTML}{FFBA00}
\definecolor{LightGrayishPink}{HTML}{EEC5D5}
\definecolor{VerySoftBlue}{HTML}{B5AFDB}
\begin{document}
\title{Spanning trees in the square of pseudorandom graphs}
\date{}
\author{Mat\'ias Pavez-Sign\'e\thanks{Supported by ANID Basal Grant CMM FB210005 and by the European
Research Council (ERC) under the European Union Horizon 2020 research and innovation programme (grant agreement No. 947978) while the author was affiliated to the University of Warwick. Centro de Modelamiento Matem\'atico (CNRS IRL2807), Universidad de Chile, Santiago, Chile. Email: \texttt{mpavez@dim.uchile.cl}}}

\maketitle
\begin{abstract}
  We show that for every $\Delta\in\mathbb N$, there exists a constant $C$ such that if $G$ is an $(n,d,\lambda)$-graph with $d/\lambda\ge C$ and $d$ is large enough, then $G^2$ contains every $n$-vertex tree with maximum degree bounded by $\Delta$. This answers a question of Krivelevich.  \end{abstract}
\section{Introduction}
A pseudorandom graph $G$ on $n$ vertices is a sparse graph that ``resembles'' many of the properties that are typically present in the binomial random graph $G(n,p)$ with edge density $p=e(G)/\binom{n}{2}$. Arguably, the most crucial characteristic of random graphs that pseudorandom graphs try to capture is the \textit{uniform edge distribution} property, that is, that all large subsets of vertices span approximately the expected number of edges that appear in the truly random case. 

In this paper, we will take a widely used approach to pseudorandom graphs based on a \textit{spectral gap} condition. Say that a graph $G$ is an  $(n,d,\lambda)$-graph if $G$ is an $n$-vertex $d$-regular graph such that all of the non-trivial eigenvalues of $G$ are bounded by $\lambda$ in absolute value, in which case, the so-called \textit{expander mixing lemma} implies that $G$ enjoys the uniform edge distribution property. We refer to the excellent  paper by Krivelevich and Sudakov~\cite{krivelevich2006pseudo} for a comprehensive survey about pseudorandom graphs. In this article, we are interested in the following extremal question.
\begin{question}\label{question:main}
     For an $n$-vertex graph $H$, how large $d/\lambda$ must be so that every  $(n,d,\lambda)$-graph contains a copy of $H$. 
\end{question}
One of the most important directions here is when $H$ is a Hamilton cycle, in which case we have the following beautiful conjecture posed by Krivelevich and Sudakov nearly 20 years ago.
\begin{conjecture}[Krivelevich and Sudakov~\cite{KrivelevichHamilton}]\label{conjecture:Hamilton}There exists a positive constant $C$ such that the following holds. If $G$ is an $(n,d,\lambda)$-graph with $d/\lambda\ge C$, then $G$ contains a Hamilton cycle.
\end{conjecture}
Krivilevich and Sudakov~\cite{krivelevich2006pseudo} proved that $d/\lambda \ge C\log n^{1-o(1)}$ is enough to guarantee Hamiltonicity in $(n,d,\lambda)$-graphs, and quite recently Glock, Munh\'a Correira, and Sudakov~\cite{glock2023hamilton} improved this result by showing that $d/\lambda\ge C(\log n)^{1/3}$ is sufficient to force Hamiltonicity. Moreover, they showed that Conjecture~\ref{conjecture:Hamilton} is true when $d\ge n^\alpha$ for some fixed $\alpha>0$.

Besides Hamiltonicity, probably the most natural problem here is to study when $(n,d,\lambda)$-graphs contain all $n$-vertex trees with bounded maximum degree. If we believe that Conjecture~\ref{conjecture:Hamilton} is correct, then we should expect to find Hamiltonian paths in $(n,d,\lambda)$-graphs, as long as ${d}/{\lambda}$ is large enough, and, therefore, it is plausible to believe that all bounded degree spanning trees appear as well. Indeed, this was explicitly asked in 2007 by Alon,  Krivelevich, and Sudakov~\cite{alon2007embedding}. 
\begin{question}[\cite{alon2007embedding}]\label{question:alon} Is it true that for any $\Delta\in\mathbb N$, there exists a positive constant $C=C(\Delta)$ such that if $G$ is an $(n,d,\lambda)$-graph with $d/\lambda\ge C$, then $G$ contains a copy of every spanning tree $T$ with $\Delta(T)\le \Delta$.  
\end{question}
As pointed out by Glock, Munh\'a Correira, and Sudakov~\cite{glock2023hamilton}, it is not even known how to find paths of length longer than $n-O(\frac{\lambda n}{d})$ in $(n,d,\lambda)$-graphs when $d/\lambda\ge C$ for some large constant $C$. Therefore, looking for general spanning trees in optimal pseudorandom graphs seems to be a quite challenging problem.  For almost-spanning trees, however, Alon, Krivelevich, and Sudakov~\cite{alon2007embedding} showed that for any $\Delta\in\mathbb N$ and $\varepsilon>0$, there exists a constant $C=C(\varepsilon,\Delta)$ such that if $G$ is an $(n,d,\lambda)$-graph with $d/\lambda\ge C$, then $G$ contains a copy of each tree with maximum degree bounded by $\Delta$ and at most $(1-\varepsilon)n$ vertices. The dependency of the constant $C=C(\varepsilon,\Delta)$ was improved in subsequent works by Balogh, Csaba, Pei, and Samotij~\cite{Balogh2010} and by Montgomery, Pavez-Sign\'e, and Yan~\cite{MPY2023} when $\varepsilon>0$ is sufficiently small. Regarding spanning trees, an unpublished work of Dellamonica Jr~\cite{dellamonicaspanning} shows that $(n,d,\lambda)$-graphs contain a specific type of spanning tree of maximum degree $3$ when $d/\lambda$ is large, and a recent result by Han and Yang~\cite{han2022spanning} establishes that $d\ge 2\lambda\Delta^{5\sqrt{\log n}}$ is enough for an $(n,d,\lambda)$-graph to contain all $n$-vertex trees with maximum degree bounded by $\Delta$.

A new twist to this problem was recently introduced by Krivelevich~\cite{krivelevich2023crowns}, who considered a weakened version of Question~\ref{question:main} by replacing the $(n,d,\lambda)$-graph $G$ with its square\footnote{The square of a graph $G$, denoted $G^2$, is the graph obtained from $G$ by adding edges between every pair of vertices at distance $2$ in $G$.}. Krivelevich~\cite{krivelevich2023crowns} proved that if $G$ is an $(n,d,\lambda)$-graph with $d/\lambda\ge C$ for some large constant $C$, then it contains a spanning subgraph which consists of a linear length cycle together with linearly many non-adjacent leaves attached to it, which he called \textit{crown}, which implies that $G^2$ is Hamiltonian. He also asked whether a  similar result for spanning trees holds or not.
\begin{question}[\cite{krivelevich2023crowns}]\label{problem:krivelevich}Is it true that for every $\Delta\in\mathbb N$, there exists a positive constant $C=C(\Delta)$ such that if $d/\lambda \ge C$ and $T$ is an $n$-vertex tree with $\Delta(T)\le \Delta$, then the square $G^2$ of an $(n,d,\lambda)$-graph $G$ contains a copy of $T$. \end{question}
Our first result is a positive answer to Question~\ref{problem:krivelevich}. 
\begin{theorem}\label{thm:main}For every $\Delta\in\mathbb N$, there exists a positive constant $C$ such that the following holds for every sufficiently large $d\in\mathbb N$. If $G$ is an $(n,d,\lambda)$-graph with $d/\lambda \ge C$, then $G^2$ contains a copy of every $n$-vertex tree with maximum degree at most $\Delta$.   
\end{theorem}
A well-known result says that trees contain either a large collection of leaves or many vertex-disjoint induced paths of some fixed length, in compensation. If we are given a tree $T$ with few leaves, and therefore many long induced paths, we can then define a new tree $\tilde T$ which is obtained from $T$ by replacing a single edge from each of those paths in $T$ with a \textit{spike}. This new tree $\tilde{T}$ has the following two main features. Firstly, $\Tilde{T}$ has bounded maximum degree and contains many leaves and, secondly, if $G$ contains a copy of $\Tilde{T}$, then $G^2$ contains a copy of $T$.  Therefore, we will deduce Theorem~\ref{thm:main} from the following result, which might be of independent interest.
\begin{theorem}\label{thm:manyleaves} For every $\Delta\in\mathbb N$ and $\alpha>0$, there exists a positive constant $C$ such that the following holds for every sufficiently large $d\in\mathbb N$. If $G$ is an $(n,d,\lambda)$-graph with $d/\lambda\ge C$, then $G$ contains a copy of every $n$-vertex tree with maximum degree at most $\Delta$ and at least $\alpha n$ leaves.
\end{theorem}
The paper is organised as follows. In Section~\ref{section:overview}, we give an overview of the proof of Theorem~\ref{thm:manyleaves}. In Section~\ref{section:preliminaries} we introduce the main tools that we need here, and we prove Theorems~\ref{thm:main} and~\ref{thm:manyleaves} in Section~\ref{section:proofs}. Lastly, we give some concluding remarks in Section~\ref{section:conclusion}. 
\section{Outline of the proof of Theorem~\ref{thm:manyleaves}}\label{section:overview}
Suppose we are given an $(n,d,\lambda)$-graph $G$ and a tree $T$ which contains a set of leaves $L$ of size $|L|\ge \alpha n$. In order to embed $T$, we will follow a similar approach as it has been done before for trees with many leaves (see~\cite{Pedro,han2022spanning,krivelevich2010embedding,montgomery2019spanning} for instance). Roughly speaking, the idea is to first embed $T-L$ and then find a matching between the image of the parents of $L$ and the unoccupied vertices in $G$. 

Assume for a moment that we can actually embed $T-L$ and let us discuss how to complete the embedding of $T$. When the host graph is a truly random graph, this can be easily done by just \textit{sprinkling} a few more edges and then showing that with high probability there exists a matching between the set of parents and the rest of the uncovered vertices in the graph. However, when working with pseudorandom graphs, this strategy is not possible anymore. To overcome this issue, we will use the idea of \textit{matchmakers} as introduced by Montgomery~\cite{montgomery2019spanning} and recently implemented by Krivelevich~\cite{krivelevich2023crowns}.

We first pick pairwise disjoint small random subsets $V_1,V_2,V_3\subset V(G)$, called matchmakers, and show that with positive probability every vertex in $V(G)$ has $\Omega(\lambda )$ neighbours in each of the $V_i$'s (this is done by using the {Lov\'asz's local lemma}). This property will then imply that small sets of vertices expand into each of the $V_i$'s (see Lemma~\ref{lemma:expansion}). We will use each of these sets $V_1,V_2$ and $V_3$ for different purposes. Firstly, we use $V_3$ to show that even after removing $V_1$ from $G$, we still have good expansion properties. Secondly, we prove that if we embed $T-L$ outside $V_1$, then we can use the properties of $V_1$ to show that the image of the set of parents of $L$ will expand into the set of unoccupied vertices in $G$. Lastly, we will use $V_2$ to show that the set of unused vertices in $G$ also have good expansion properties in the image of the parents of $L$. To perform this last step, however, we need to embed $T-L$ while ensuring that the image of the parents of $L$ covers every vertex from $V_2$. We will explain now how to do this.  

The main tool that we use to embed trees is a powerful embedding technique, sometimes called \textit{extendability methods} or \textit{tree embeddings with rollbacks}, which was first introduced by Friedman and Pippenger~\cite{FP1987} in 1987 and subsequently improved by Haxell~\cite{H2001} in 2001. Here we will use a modern reformulation of this technique which is attributed to Glebov, Johannsen, and Krivelevich~\cite{Glebov2013}, and that has played a major role in the solution of several problems in the last few years (see~\cite{Pedro, araujo2022ramsey,draganic2022rolling,glock2023hamilton,han2022spanning, krivelevich2023turan,krivelevich2023crowns, montgomery2019spanning, MPY2023} for instance).  Roughly speaking, the extendability method (Lemma~\ref{lemma:adding:leaf}) says that if we are given a subgraph $S_i\subset G$ which is 'extendable' and $G$ has good expansion properties, then we can extend $S_i$ by adding an edge $e_i$ with one of its endpoints in $V(S_i)$ and other in $V(G)\setminus V(S_i)$ so that $S_i+e_i$ remains extendable. In $(n,d,\lambda)$-graphs, this method works smoothly as long as $|S_i|\le |G|-\Theta(\frac{\Delta\lambda n}d)$, and therefore, since $|L|\ge \alpha n\gg \frac{\lambda n}{d}$, we will be able to iterate this process until we embed all of $T-L$. The main issue here, however, is that we need to ensure that the matchmaker $V_2$ is completely contained in the image of the parents of $L$.

In order to cover $V_2$, we will use some further ideas from the work of Montgomery~\cite{montgomery2019spanning}. We first take a large set $Q$ of parents of leaves which are far apart from each other in the tree (this is possible as $T$ has bounded degree). Using extendability methods, we will embed $T'$ in rounds so that at each round we cover more and more of $V_i$ using only vertices from $Q$ at each step. After this is completed, every vertex from $V_2$ will be covered by vertices from $Q$ and then we just finish the embedding of $T-L$ using extendability methods. To complete the embedding of $T$, we use Hall's theorem to find a matching between the image of the parents of leaves and the leftover vertices in the graph. The properties of the matchmakers will guarantee that Hall's matching criteria is satisfied and thus we can complete the embedding of $T$. 
\section{Preliminaries}\label{section:preliminaries}
We will use standard graph theory notation. For a graph $G$, we denote by $V(G)$ and $E(G)$ the set of vertices and edges of $G$, respectively, and write $|G|=|V(G)|$ and $e(G)=|E(G)|$. For a vertex $v\in V(G)$, we denote by $N(v)$ the set of neighbours of $v$ and let $d(v)=|N(v)|$ denote the degree of $v$. Given a subset $S\subset V(G)$, the set of neighbours of $S$ is $\Gamma(S)=\bigcup_{s\in S}N(s)$ and the external neighbourhood of $S$ is $N(S)=\Gamma(S)\setminus S$. For a vertex $v\in V(G)$ and sets $U,S\subset V(G)$, we write $d(v,U)=|N(v)\cap U|$ and $N(S,U)=N(S)\cap U$. When working with more than one graph, we will use a subscript to specify which graph are we working with. For example, if $H$ is a subgraph of $G$ and $v\in V(H)$, then $d_H(v)$ denotes the degree of $v$ in $H$. Given a subset $S\subset V(G)$, we write $G[S]$ to denote the graph \textit{induced} by $S$, that is, the graph with vertex set $S$ and all the edges from $G$ with both endpoints in $S$, and we write $G-S$ for the graph $G[V(G)\setminus S]$. For two sets $A,B\subset V(G)$, we let $e(A,B)$ denote the number of edges with one endpoint in $A$ and the other endpoint in $B$, and, if $A$ and $B$ are disjoint, we let $G[A,B]$ denote the bipartite graph induced by $A$ and $B$ in $G$, in which case $e(A,B)$ is just the number of edges in $G[A,B]$. For a graph $H$ and an edge $e\not\in E(H)$, we let $H+e$ denote the graph obtained from $H$ by adding the edge $e$. 

For $n\in\mathbb N$, we write $[n]=\{1,\dots,n\}$. We will use the standard hierarchy notation, that is, for real numbers $a,b\in (0,1]$, we will write $a\ll b$ to mean that given $b$ fixed, there exists $a_0>0$ such that if $a\le a_0$ then all the subsequent relevant statements hold. If $1/x$ appears in such a hierarchy, we will assume that $x$ is a natural number, and hierarchies with more constants are defined in a similar way and are to be read from right to left. 
\subsection{Probabilistic tools}
We will use the following standard probabilistic results (see \cite[Corollary 2.3]{JLR2000} and \cite[Corollary 5.1.2]{alon2016probabilistic}).

\begin{lemma}[Chernoff's bound]\label{lemma:chernoff}Let $X$ be binomial random variable. Then, for all $0<\varepsilon\le \frac{3}{2}$, 
\[\mathbb P(\big|X-\mathbb E[X]\big|\ge \varepsilon \mathbb E[X])\le 2\exp \left(-\frac{\varepsilon^2}{3}\mathbb E[X]\right).\]
\end{lemma}
\begin{lemma}[Lov\'asz's local lemma]\label{lemma:local}Let $A_1,\ldots, A_n$ be events in a probability space. Suppose that each event $A_i$ is independent of all the other events $A_j$ but at most $d$. If $\mathbb P(A_i)\le p$ for all $i\in [n]$ and $ep(d+1)\le 1$, then $\mathbb P(\bigwedge_{i=1}^n\overline{A_i})>0$.
    
\end{lemma}
\subsection{Dividing trees}

Given a tree $T$, say that two subtrees $S_1,S_2\subset T$ divide $T$ if $S_1$ and $S_2$  share exactly one vertex and $T=S_1\cup S_2$.
\begin{lemma}[{\cite[Proposition 3.19]{montgomery2019spanning}}]\label{lemma:divide:trees} Let $T$ be a tree and let $Q\subset V(T)$ be a fixed subset. Then, there exists subtrees $S_1$ and $S_2$ that divide $T$ and $|V(S_1)\cap Q|,|V(S_2)\cap Q|\ge |Q|/3$.
\end{lemma}
For a tree $T$, say that a subset $X\subset V(T)$ is $k$-separated if every pair of vertices from $X$ are at distance at least $k$ in $T$. The following result says that large subsets of bounded degree trees contain a large separated subset. 
\begin{lemma}[{\cite[Corollary 3.16]{montgomery2019spanning}}]\label{lemma:separated}Let $\Delta\in\mathbb N$ and $k\ge 0$. Let $T$ be a tree with $\Delta(T)\le\Delta$ which contains a subset $X\subset V(T)$ of size $|X|\ge 3\Delta^k$. Then, there exists a subset $Q\subset X$ which is $(2k+2)$-separated in $T$ and $|Q|\ge |X|/(8k+8)\Delta^k$.
\end{lemma}

\subsection{Expansion properties of pseudorandom graphs}
In this section, we collect properties of $(n,d,\lambda)$-graphs that we will use throughout the paper. The main ingredient that we use is the well-known expander mixing lemma (see \cite[Theorem 2.11]{krivelevich2006pseudo} for a proof).
\begin{lemma}[Expander Mixing Lemma]\label{lemma:mixing}Let $G$ be an $(n,d,\lambda)$-graph. Then, for every pair of (not necessarily disjoint) sets $A,B\subset V(G)$, we have 
\[\left|e(A,B)-\tfrac{d}{n}|A||B|\right|<\lambda\sqrt{|A||B|}.\]
\end{lemma}
\begin{definition}Say that a graph $G$ is $m$-joined if for every pair of disjoint subsets $A,B\subset V(G)$, each of size $m$, there exists at least one edge between them.\end{definition}
A direct consequence of the expander mixing lemma is that  $(n,d,\lambda)$-graphs are $\frac{\lambda n}{d}$-joined.
\begin{corollary}\label{corollary:joined}If $G$ is an $(n,d,\lambda)$-graph, then $G$ is $\frac{\lambda n}d$-joined.    
\end{corollary}
\begin{proof}Given two disjoint subsets $A,B\subset V(G)$ of size $|A|=|B|=\lambda n/d$, by the Expander Mixing Lemma we have
\[e(A,B)> \frac{d}{n}|A||B|-\lambda\sqrt{|A||B|}=0,\]
as $|A||B|=(\lambda n/d)^2$.
\end{proof}
The following lemma states that in $(n,d,\lambda)$-graphs we can translate minimum degree conditions into an expansion property for small sets. 
\begin{lemma}\label{lemma:expansion}Let $D\ge 1$ and let $d,\lambda>0$ satisfy $d/\lambda>2D$. Let $G$ be an $(n,d,\lambda)$-graph which contains a subset $X\subset V(G)$ such that every vertex $v\in V(G)$ has at least $2D\lambda$ neighbours in $X$. Then, every subset $S\subset V(G)$ of size $|S|\le \lambda n/d$ satisfies $|N(S)\cap X|\ge D|S|$.
\end{lemma}
\begin{proof}Suppose that there exists a subset $S\subset V(G)$  of size $1\le |S|\le \lambda n/d$ such that $|N(S)\cap X|<D|S|$. Let $Y=N(S)\cap X$. Using the Expander Mixing Lemma, we have
    \[2\lambda D|S|\le e(S,Y)<\frac{d}{n}|S|\cdot |Y| +\lambda\sqrt{|S||Y|}\le \lambda \cdot D|S|+\lambda |S|\sqrt{D} ,\]
    and thus 
    \[2\lambda D<\lambda D+\lambda \sqrt{D}\le 2\lambda D,\]
    which is a contradiction.
\end{proof}
The last result we need is a lower bound on the second largest eigenvalue in $(n,d,\lambda)$-graphs.
\begin{lemma}\label{lemma:secondeig}Every $(n,d,\lambda)$-graph satisfies $\lambda\ge \sqrt{d\cdot\frac{n-d}{n-1}}$.    
\end{lemma}
\subsection{Finding matchmakers}

The following result finds, in an $(n,d,\lambda)$-graph $G$, a collection of pairwise disjoint small subsets $V_1,\ldots, V_\ell\subset V(G)$ so that each vertex in $G$ has a large number of neighbours in each of the $V_i$'s.
\begin{lemma}\label{lemma:matchmaker}Let $1/d\ll 1/C\ll1/t\ll 1/\ell$ and let $d\in\mathbb N$ and $\lambda>0$ satisfy $d/\lambda\ge C$. Let $n\in\mathbb N$ satisfy $d\le 3n/5$ and let $G$ be an $(n,d,\lambda)$-graph. Then, there exist disjoint subsets $V_1,\ldots, V_\ell \subset V(G)$, each with at most $4\lambda tn$ vertices, such that every vertex $v\in V(G)$ has at least $t\lambda/4$ neighbours in each of the $V_i$'s.
\end{lemma}
\begin{proof}Letting $k=\lfloor{d/t\lambda}\rfloor$, note that $k\ge 1$ as $d/\lambda\ge C>t$. We colour each vertex from $G$ uniformly at random with an element from $[k]$, making all choices independently.  For $v\in V(G)$ and a colour $i\in [k]$, let $A_{v,i}$ denote the event that $v$ has less than $d/2k$ neighbours in colour $i$. Then, using Lemma~\ref{lemma:chernoff}, we have
\[\mathbb P(A_{v,i})\le 2\exp (-d/12k).\]
Note that the event $A_{v,i}$ is not independent only of those events $A_{u,j}$ such that either $u$ and $v$ have common neighbours or $u=v$, and the number of such events is at most $2kd^2$. Using Lemma~\ref{lemma:secondeig} and that $d\le 3n/5$, we get 
\[kd^2e^{-d/12k}\le td\lambda e^{-\lambda t/24}\le d^2e^{-t\sqrt{d}/100}\ll 1\]
whenever $d$ is large enough. Therefore, from Lemma~\ref{lemma:local} we deduce that $\mathbb P(\bigwedge_{i\in [n]}\overline{A_i})>0$, and thus there is a colouring of $V(G)$ so that every vertex has at least $d/2k\ge t\lambda /4$ neighbours in each colour. Let $V_1,\ldots, V_\ell$ be the $\ell$ smallest colour classes and suppose that $|V_1|\le \ldots\le |V_\ell|$. Then, using that $\ell \le k/2$ if $t$ is large enough, we get 
\[|V_i|\le |V_\ell|\le \frac{n}{k-\ell}\le \frac{2n}{k}\le \frac{4\lambda tn}{d}.\]
    
\end{proof}
\subsection{The extendability method}\label{section:extendability}
Here we state some of the main tools of the extendability method, we refer to~\cite{Glebov2013,montgomery2019spanning} for a more comprehensive exposition of this technique.

\begin{definition}Let $D,m\in\mathbb N$ be such that $D\ge 3$. For a graph $G$ and a subgraph $S\subset G$, say that $S$ is $(D,m)$-extendable in $G$ if $S$ has maximum degree at most $D$ and for all $U\subset V(G)$ with $1\le |U|\le 2m$ one has
    \begin{equation}\label{def:extendability}
        |\Gamma_G(U)\setminus V(S)|\ge (D-1)|U|-\sum_{u\in U\cap V(S)}(d_S(u)-1).
    \end{equation}
\end{definition}
The following lemma says that it is enough to control the external neighbourhood of small sets in order to verify extendability.

\begin{proposition}[{\cite[Proposition 3.2]{montgomery2019spanning}}]\label{prop:weakerextendability}
Let $D,m\in\mathbb N$ satisfy $D\ge 3$ and $m\ge 1$. Let $G$ be a graph and let $S\subset G$ be a subgraph with $\Delta(S)\le D$. If for all $U\subset V(G)$, with $1\leq|U|\leq 2m$, we have 
\[|N(U,V(G)\setminus V(S))|\geq D|U|,\]
then $S$ is $(D,m)$-extendable in $G$.
\end{proposition}

The next result states that we can add leaves to an extendable subgraph while remaining extendable. 

\begin{lemma}[{\cite[Corollary 3.5]{montgomery2019spanning}}]\label{lemma:adding:leaf}
Let $D,m\in\mathbb N$ be such that $D\ge 3$, and let $G$ be an $m$-joined graph. Let $S$ be a $(D,m)$-extendable subgraph of $G$ such that $|G|\ge |S|+(2D+3)m+1$. Then for every $s\in V(S)$ with $d_S(s)\leq D-1$, there exists $y\in N_G(s)\setminus V(S)$ such that $S+sy$ is $(D,m)$-extendable.
\end{lemma}
A direct consequence of Lemma~\ref{lemma:adding:leaf} is that we can embed large trees and remain extendable.
\begin{corollary}[{\cite[Corollary 3.7]{montgomery2019spanning}}]\label{corollary:extendable:embedding}
Let $D,m\in\mathbb N$ be such that $D\ge 3$, and let $G$ be an $m$-joined graph. Let $T$ be a tree with $\Delta(T)\le D/2$ and let $H$ be a $(D,m)$-extendable subgraph of $G$ with maximum degree at most $D/2$. If $|H|+|T|\le|G|-(2D+3)m$, then for every vertex $t\in V(T)$ and $v\in V(H)$, there is a copy $S$ of $T$ in $G-V(H-v)$ in which $t$ is copied to $v$ and, moreover, $S\cup H$ is a $(D,m)$-extendable subgraph of $G$.
\end{corollary} Given a graph $G$ and a subset $X\subset V(G)$, we let $I(X)$ denote the independent subgraph induced by $X$, that is, the subgraph of $G$ with vertex set $X$ and no edges. The following lemma is a covering result due to Montgomery~\cite{montgomery2019spanning}.
\begin{lemma}[{\cite[Lemma 4.1]{montgomery2019spanning}}]\label{lemma:covering:1}
Let $k,D,m\in\mathbb N$ with $D\ge 20$. Let $G$ be an $m$-joined graph and  $H\subset G$ be a subgraph with $\Delta(H)\le D/4$. Let $X\subset V(G)\setminus V(H)$ be a subset such that $H\cup I(X)$ is $(D,m)$-extendable in $G$ and let $T$ be a tree with $\Delta(T)\le D/4$ that satisfies $|H|+|X|+|T|\le |G|-10Dm-2k$. 

Suppose that $Q\subset V(T)$ is a $(4k+4)$-separated set in $T$ which satisfies $|Q|\ge 3|X|$, and let $t\in V(T)$ and $v\in V(H)$. Then, there is a copy $S$ of $T$ in $G-V(H-v)$ so that $t$ is copied to $v$, $H\cup I(X)\cup S$ is $(D,m)$-extendable in $G$, $|X\setminus V(S)|\le 2m/(D-1)^k$, and all vertices in $X\cap V(S)$ have a vertex in $Q$ copied to them.
\end{lemma}

\section{Proofs}\label{section:proofs} The last ingredient we need is the well-known Koml\'os--S\"ark\'ozy--Szemer\'edi theorem about spanning trees in dense graphs, which will allow us to restrict the proof to sparser graphs.
\begin{theorem}[\cite{KSS95}]\label{thm:KSS}For every $\Delta\in\mathbb N$ and $\varepsilon>0$, there is $n_0\in\mathbb N$ such that for all $n\ge n_0$ the following holds. If $G$ is an $n$-vertex graph with $\delta(G)\ge (1+\varepsilon)\frac{n}{2}$, then $G$ contains a copy of each tree with $n$ vertices and maximum degree at most $\Delta$.    
\end{theorem}
Instead of proving Theorem~\ref{thm:manyleaves} directly, we will show a more general result (Theorem~\ref{thm:general} below) from which Theorem~\ref{thm:manyleaves} is a corollary. Nevertheless, up to small modifications, the proof of Theorem~\ref{thm:general} follows closely the sketch given in Section~\ref{section:overview} for the proof of Theorem~\ref{thm:manyleaves}.
\begin{theorem}\label{thm:general}For every $\Delta\in\mathbb N$, there exist positive constants $K$ and $C$ such that the following holds for every sufficiently large $d\in \mathbb N$. If $G$ is an $(n,d,\lambda)$-graph with $d/\lambda\ge C$, then $G$ contains a copy of every $n$-vertex tree with maximum degree at most $\Delta$ and at least $\frac{K\lambda n}d$ leaves.    
\end{theorem}
\begin{proof}
We start by fixing the constants
\[1/d\ll 1/C\ll 1/K\ll 1/t\ll 1/\Delta.\]
Let $\lambda>0$ satisfy $d/\lambda\ge C$, let $G$ be an $(n,d,\lambda)$-graph and let $T$ be an $n$-vertex tree with $\Delta(T)\le \Delta$ which contains at least $\frac{K\lambda n}{d}$ leaves. Using Theorem~\ref{thm:KSS}, we may assume that $d\le 3n/5$, which implies, by Lemma~\ref{lemma:secondeig}, that 
\[\lambda\ge \sqrt{d\cdot \frac{n-d}{n-1}}\ge \frac{\sqrt{d}}{2}.\]
\noindent \textbf{Step 1. Setting the matchmakers}: Apply Lemma~\ref{lemma:matchmaker} to find pairwise disjoint sets $V_1,V_2,V_3\subset V(G)$ such that 
\begin{enumerate}[label = \upshape\textbf{A\arabic{enumi}}]
\item\label{A:1} $|V_1|,|V_2|,|V_3|\le \frac{4t\lambda n}{d}$, and 
    \item \label{A} for every $v\in V(G)$ and $i\in [3]$, $d(v,V_i)\ge \frac{t\lambda}{4}$.
\end{enumerate}
Set $V'=V(G)\setminus V_1$ and $G'=G[V']$. Let $D=\frac{t}{10^3}$, and note that $D$ is much larger than $\Delta$. For a fixed vertex $v_0\in V'\setminus (V_2\cup V_3)$, we claim that $I(V_2\cup \{v_0\})$ is $(D,\frac{\lambda n}d)$-extendable in $G'$. Indeed, given a subset $S\subset V(G')$ of size $1\le |S|\le \frac{\lambda n}d$, using~\ref{A} and Lemma~\ref{lemma:expansion}, we deduce that 
\[|N_{G'}(S)|\ge |N_G(S)\cap V_3|\ge 10D|S|.\]
On the other hand, if $\frac{\lambda n}{d}\le |S|\le \frac{2\lambda n}{d}$, we then  have
\[|N_{G'}(S)|\ge |N_G(S)\cap V_3|\ge 10D\cdot \frac{\lambda n}{d}-|S|\ge 8D\cdot \frac{\lambda n}{d}\ge 4D|S|,\]
and therefore, by Lemma~\ref{prop:weakerextendability}, $I(V_2\cup\{v_0\})$ is $(D,\frac{\lambda n}{d})$-extendable in $G'$.\\

\noindent\textbf{Step 2. Covering the matchmaker:} Let $t\in V(T)$ be an arbitrary vertex which is not a leaf nor a parent of leaves (this is clearly possible as $\Delta(T)\le \Delta$ and $n$ is large). Letting $P$ denote the set of parents of leaves of $T$, we have that
\[|P|\ge\frac{K\lambda n}{d\Delta }\ge \frac{K \sqrt{n}}{2\Delta}\ge 3\Delta^5,\]
as $\lambda \ge \frac{1}{2}\sqrt{d}$ and $n\ge d\ge C$. Use Lemma~\ref{lemma:separated} to find a $12$-separated set $Q\subset P$ of size 
\[|Q|\ge \frac{|P|}{48\Delta^5}\ge \frac{K\lambda n}{50d\Delta^5}.\]
Let $L\subset V(T)$ be the set of leaves of $T$ and set $T'=T-L$. Our goal now is to find a copy of $T'$ in $G'$ so that every vertex from $V_2$ is covered by a vertex from $Q$. Let 
\[\ell=\left\lfloor\log_{D-1}(\tfrac{2\lambda n}d)\right\rfloor+1.\]
\begin{claim}There is a sequence of subtrees $T_1,\ldots, T_\ell\subset T'$ and a sequence of vertices $t_1,\ldots, t_\ell$, with $t_1=t$, such that 
\begin{enumerate}[label = \upshape\textbf{B\arabic{enumi}}]
    \item \label{B:1}$T'=T_1\cup\ldots\cup T_\ell$, 
    \item $t_1\in V(T_1)$,
    \item \label{B:2}$T_i$ intersects $T_{i+1}$ at $t_{i+1}$, and
    \item \label{B:3}for each $i\in [\ell]$, there is a $(4i+8)$-separated set $Q_i\subset V(T_i)$  which does not contain $t_{i}$ and has size $|Q_i|\ge \frac{K\lambda n}{150d\Delta^5(D-1)^{i+1}}$.
\end{enumerate} 
\end{claim}
\begin{proof}[Proof of claim] We say that a sequence of subtrees $T_1,\ldots, T_{i-1},S_i\subset T'$ and (non necessarily distinct) vertices $t_1,\ldots, t_i\in V(T')$ is a good sequence of length $i$ if the following properties hold:
\begin{enumerate}[label=(\roman{enumi})]
\item $t_1\in V(T')$.
    \item $T'=T_1\cup\ldots\cup T_{i-1}\cup S_i$.
    \item For $1\le j\le i-1$, $T_j$ intersects $T_{j-1}$ exactly at $t_{j}$, and $S_i$ intersects $T_{i-1}$ exactly at $t_i$.
    \item\label{iteration:iv} $|V(S_i)\cap Q|\ge |Q|/3^{i-1}$ and $|V(T_j)\cap Q|\ge |Q|/3^j$ for $j\in [i-1]$.
\end{enumerate}
First, note that $T_1=\emptyset$, $S_1=T'$, and $t_1=t$ form a good sequence of length $1$. We now show that we can find a good sequence of length $\ell$. Suppose that, for some $1\le i<\ell $, we have found a good sequence $T_1,\ldots,T_{i-1},S_i\subset T'$ and $t_1,\ldots, t_i\in V(T')$. Use Lemma~\ref{lemma:divide:trees} to find subtrees $T_i,S_{i+1}\subset S_i$ such that $S_{i+1}$ and $T_i$ divide $S_i$ and $|V(S_{i+1})\cap Q|,|V(T_i)\cap Q|\ge |V(S_i)\cap Q|/3$. Moreover, we may assume that $t_i\in V(T_i)$ and let $t_{i+1}$ be the unique vertex in $V(S_i)\cap V(T_i)$. Finally, from~\ref{iteration:iv} we deduce that 
\[|V(S_{i+1})\cap Q|,|V(T_i)\cap Q|\ge \frac{|V(S_i)\cap Q|}{3}\ge \frac{|Q|}{3^i}.\]
This implies that $T_1,\ldots, T_i,S_{i+1}$ and $t_1,\ldots, t_{i+1}$ is a good sequence of length $i+1$, and thus, after $\ell$ steps, we can find a good sequence of length $\ell$. Let $T_1,\ldots,T_{\ell-1},S_\ell$ and $t_1,\ldots, t_\ell$ be such a sequence and set $T_\ell=S_{\ell}$. We claim that this sequence satisfies~\ref{B:1}--\ref{B:3}. Indeed, let $Q_1=Q\cap V(T')$ and note that $Q_1$ is $12$-separated, as $Q$ is $12$-separated, and has size $|Q_1|\ge |Q|/3\ge \frac{K\lambda n}{150d\Delta^5(D-1)}$. For each $2\le i\le \ell$, use Lemma~\ref{lemma:separated} to find a $(4i+8)$-separated set $Q_i\subset V(T_i)$ of size 
\begin{equation*}\label{equation:Q_i}|Q_i|\ge \frac{|Q\cap V(T_i)|}{(16i+32)\Delta^{2i+3}}\ge \frac{K\lambda n}{50d\Delta^5\cdot 3^i(16i+32)\Delta^{2i+3}}\ge \frac{K\lambda n}{50d\Delta^5(D-1)^{i+1}},\end{equation*}
where the last inequality holds as $D$ is sufficiently large compared to $\Delta$. By potentially removing $t_i$ from $Q_i$, we can assume that $t_i\not\in Q_i$ and $|Q_i|\ge \frac{K\lambda n}{150d\Delta^5(D-1)^{i+1}}$, as required.
\end{proof}
Now we find the copy of $T'$ while covering every vertex from $V_2$. For $1\le i\le \ell$, say that we have a \textit{Stage~$i$ situation} if we have a subgraph $F_i\subset G'$ and a subset $X_i\subset V_2$, disjoint from $F_i$, such that the following properties hold:
\begin{enumerate}[label = \upshape\textbf{C\arabic{enumi}}]
    \item\label{C:1}$F_i$ is a copy of $T_1\cup\ldots\cup T_i$ with $t$ copied to $v_0$.
    \item\label{C:2} $F_i\cup I(X_i)$ is $(D,\frac{\lambda n}{d})$-extendable.
    \item\label{C:3} $V_2\cap V(F_i)\subset Q_1\cup\ldots\cup Q_i$ and $|V_2\setminus V(F_i)|\le \frac{2\lambda n}{d(D-1)^{i+1}}$.
\end{enumerate}
Let us first produce a Stage 1 situation. Firstly, from~\ref{B:3}, we have that $Q_1$ is a $12$-separated set in $T'$ of size 
\[|Q_1|\ge \frac{K\lambda n}{150d\Delta^5(D-1)}\ge 3|V_2|,\]
as $|V_2|\le \frac{4t\lambda n}d$ by~\ref{A:1} and $1/K\ll 1/t,1/\Delta$. Secondly, since
\begin{equation}\label{eq:iteration}|T'|+|V_2|\le n-\frac{K\lambda n}{d}+\frac{4t\lambda n}{d}\le |G'|-20D\cdot \frac{\lambda n}{d},\end{equation}
as $D=\frac{t}{10^3}$ and $1/K\ll 1/t$, we can use Lemma~\ref{lemma:covering:1} to find a copy $F_1$ of $T_1$ in $G'$ with $t$ copied to $v_0$ such that, if $X_1=V_2\setminus V(F_1)$, we have that $I(X_1)\cup F_1$ is $(D,\frac{\lambda n}{d})$-extendable in $G'$ and $|X_1|\le 2\lambda n/d(D-1)^2$. Lastly, every vertex in $V_2\cap V(F_1)$ is covered by some vertex of $Q_1$, which proves that we have a Stage $1$ situation. Assume that we have a Stage $i$ situation, for some $1\le i<\ell$, and let us show how to produce a Stage $i+1$ situation. Let $F_i\subset G'$ and $X_i\subset V_2$ satisfy~\ref{C:1}--\ref{C:3}. Again, from~\ref{B:3}, we have a $(4i+8)$-separated set $Q_{i+1}\subset V(T_{i+1})$ which does not contain $t_{i+1}$ such that 
\[|Q_{i+1}|\ge \frac{K\lambda n}{100d\Delta^5(D-1)^{i+1}}\ge\frac{6\lambda n}{d(D-1)^{i+1}}\ge 3|X_i|,\]
where we used~\ref{C:3} and that $X_i\subset V_2\setminus V(F_i)$. Then, since $T_1\cup\ldots\cup T_{i+1}\subset T$ and $X_i\subset V_2$, equation~\eqref{eq:iteration} implies that can use Lemma~\ref{lemma:covering:1} to find a subgraph $F_{i+1}\supset F_i$ such that~\ref{C:1} holds and, letting $X_{i+1}=X_i\setminus V(F_{i+1})$, we have that $F_{i+1}\cup I(X_{i+1})$ is $(D,\frac{\lambda n}{d})$-extendable, which shows that~\ref{C:2} holds. Moreover, every vertex in $X_i\cap V(F_{i+1})$ is covered by some vertex from $Q_{i+1}$ and $|X_{i+1}|\le \frac{2\lambda n}{d(D-1)^{i+2}}$, showing that~\ref{C:3} also holds. 

This proves that we can reach a Stage $\ell$ situation. In this scenario, we have a subgraph $F_\ell\subset G'$ which is a copy of $T_1\cup\ldots\cup T_\ell=T'$ (because of~\ref{B:1}) such that $t$ is copied to $v_0$. Moreover, by~\ref{C:3} and the definition of $\ell$, we have
\[|V_2\setminus V(F_\ell)|\le \frac{2\lambda n}{(D-1)^{\ell+1}}<1,\]
which implies that $V_2$ is covered by the image of $Q_1\cup\ldots \cup Q_\ell\subset Q$.
\\

\noindent\textbf{Step 3. Finishing the embedding:} Recall that $F_\ell$ is a copy of $T'$ so that $F_\ell$ is $(D,\frac{\lambda n}{d})$-extendable in $G'$, $t$ is copied to $v_0$, and every vertex from $V_2$ is covered by some vertex from $Q$. Recall that $P$ is a set of parents of leaves of size $|P|\ge \frac{K\lambda n}{d\Delta}$. Take a set $L'$ of leaves such that there is a perfect matching in $T$ between $P$ and $L'$, and set $T''=T-L'$. We first a copy of $T''$. Note that
\[|T'|+|T''-(T'-t)|\le n+1-\frac{2K\lambda n}{d\Delta}\le |G'|-(2D+3)\cdot \frac{\lambda n}{d},\]
as $|V_2|\le\frac{4t\lambda n}{d}$ by~\ref{A:1} and $1/K\ll 1/t$. Therefore, we can use Corollary~\ref{corollary:extendable:embedding} to find a subgraph $F\supset F_\ell$ such that 
\begin{enumerate}[label = \upshape\textbf{D\arabic{enumi}}]
    \item\label{D:1 } $F$ is a copy of $T''$, 
    \item\label{D:2} $F$ is $(D,\frac{\lambda n}{d})$-extendable in $G'$, and
    \item\label{D:3} $V_2$ is contained in the image of $P$.
\end{enumerate}
It is thus only left to embed $L'$. To do so, we only need to find a perfect between the image of $P$ and the leftover vertices in $G$, for which we will use the well-known  Hall's matching theorem. 
\begin{lemma}[Hall's matching theorem]\label{lemma:Hall}Let $H$ be a bipartite graph with parts $A$ and $B$. If for every subset $U\subset A$ we have $|N(U)|\ge |U|$, then $H$ contains a matching covering $A$.  
\end{lemma}
Let $A$ be the image of $P$ and $B=V(G-F)$, and let $H=G[A,B]$ be the bipartite graph induced by $A$ and $B$. In order to finish the embedding of $T$, we just need to check the conditions of Lemma~\ref{lemma:Hall} for $H$. Since $F$ is $(D,\frac{\lambda n}{d})$-extendable by~\ref{D:2}, for any subset $U\subset A$ with $|U|\le \frac{\lambda n}{d}$, we have
\begin{eqnarray*}|N_H(U)|\ge |\Gamma_G(U)\setminus V(F)|&\overset{\eqref{def:extendability}}{\ge}& (D-1)|U|-\sum_{u\in U\cap V(F)}(d_F(u)-1) \\
&\ge& (D-\Delta)|U|\\
&\ge &|U|.\end{eqnarray*}
For the sake of contradiction, suppose that we can find a subset $U\subset A$ with $\frac{\lambda n}{d}<|U|\le |A|$ such that $|N_H(U)|<|U|$. Firstly, note that, by Corollary~\ref{corollary:joined}, we have 
\[|U|>|N_H(U)|\ge |B|-\frac{\lambda n}{d}.\]
Secondly, let $W=B\setminus N_H(U)$ and note that, as $|A|=|B|$, we have $|A\setminus U|<|W|\le \frac{\lambda n}{d}$. Finally, by~\ref{D:3},~\ref{A} and Lemma~\ref{lemma:expansion}, we have 
\[|W|>|A\setminus U|\ge |N_H(W)|\ge |N_G(W)\cap V_2|\ge D|W|,\]
a contradiction. Therefore, we can use Lemma~\ref{lemma:Hall} to complete the embedding of $T$ and thus finish the proof. \end{proof}
\begin{proof}[Proof of Theorem~\ref{thm:manyleaves}]Given $\Delta\in\mathbb N$ and $\alpha>0$, let $K$ and $C$ from Theorem~\ref{thm:general} and assume that $C$ is large enough so that $K/C\le\alpha$. Let $d$ be sufficiently large and let $\lambda>0$ satisfy $d/\lambda\ge C$. Let $G$ be an $(n,d,\lambda)$-graph and let $T$ be an $n$-vertex tree with $\Delta(T)\le \Delta$ and at least $\alpha n\ge \frac{Kn}{C}\ge \frac{K\lambda n}{d}$ leaves. Then, $G$ contains a copy of $T$ by Theorem~\ref{thm:general}.    
\end{proof}
Given a tree $T$, say that a subgraph $P\subset T$ is a \textit{bare path} if all vertices of $P$ have degree exactly 2 in $T$. The last ingredient we need is the following structural result of trees. 
\begin{lemma}[{\cite[Lemma 2.1]{krivelevich2010embedding}}]\label{lemma:paths-leaves}Let $n,k,\ell\in\mathbb N$ and let $T$ be an $n$-vertex tree with at most $\ell$ leaves. Then, $T$ contains a collection of a least $\frac{n}{k+1}-(2\ell-2)$ vertex disjoint bare paths, each of length $k$.
\end{lemma}
\begin{proof}[Proof of Theorem~\ref{thm:main}]For $\Delta\in\mathbb N$ fixed, let $\alpha>0$ be sufficiently small and let $C>0$ and $d\in\mathbb N$ be large enough. Let $G$ be an $(n,d,\lambda)$-graph with $d/\lambda\ge C$ and let $T$ be an $n$-vertex tree with $\Delta(T)\le \Delta$. By Theorem~\ref{thm:manyleaves}, we can assume that $T$ has less than $\alpha n$ leaves. Then, Lemma~\ref{lemma:paths-leaves} implies that $T$ contains a collection of at least $\frac{n}{4}-2(\alpha n-2)\ge \frac{n}{8}$ vertex-disjoint bare paths, each of length $3$. Therefore, for $k=\lfloor{n/8}\rfloor$, we can find vertex-disjoint bare paths $P_i=a_ib_ic_id_i$, for $i\in [k]$. For each $i\in [k]$, we let $\tilde P_i$ denote the tree with vertex set $V(\tilde P_i)=\{a_i,b_i,c_i,d_i\}$ and edges $E(\tilde T_i)=\{a_ib_i,b_ic_i,b_id_i\}$ (see Figure~\ref{fig:tildeT}).

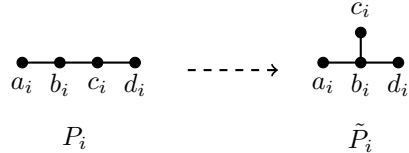
\begin{figure}[ht!]
    \centering
	\begin{tikzpicture}[scale=1]
		
		\coordinate (a) at (-2, 0);
		\coordinate (b) at (-1.5, 0);
		\coordinate (c) at (-1, 0);
		\coordinate (d) at (-.5, 0);
            \coordinate (P) at (-1.3,-.7);
    
		\draw (a) node {$a_i$};
		\draw (b) node {$b_i$};
		\draw (c) node {$c_i$};
		\draw (d)  node  {$d_i$};
            \draw (P) node {$P_i$};

		\draw (4,0)+(a) node {$a_i$};
		\draw (4,0)+(b) node {$b_i$};
		\draw (4,0)+(c) node {$d_i$};
		\draw (2.5,1)  node  {$c_i$};

            \coordinate (P') at (2.5,-.7);

            \draw (P') node {$\tilde P_i$};
    
		\draw [->,dashed, thick] (.2,.2) -- node[below] {} (1.4,.2); 
		
		\coordinate (a') at (-2,.3);
		\coordinate (b') at (-1.5,.3);
		\coordinate (c') at (-1,.3);
		\coordinate (d') at (-.5,.3);

  		\coordinate (a2) at (2,.3);
		\coordinate (b2) at (2.5,.3);
		\coordinate (c2) at (2.5,.7);
		\coordinate (d2) at (3,.3);

		\tikzstyle{every node}=[circle, draw, fill, inner sep=0pt, minimum width=4pt]
		\draw (a') node {};
		\draw (b') node {};
		\draw (c') node {};
		\draw (d') node {};

            \draw [-,thick] (a') to (b') to (c') to (d');

  		\draw (a2) node {};
		\draw (b2) node {};
		\draw (c2) node {};
		\draw (d2) node {};

            \draw [-,thick] (a2) to (b2) to (c2);
            \draw [-,thick] (b2) to (d2);

			\end{tikzpicture}
   \caption{A figure showing how we modify the bare path $P_i$.}
    \label{fig:tildeT}
\end{figure}
Let $\tilde T$ be the tree obtained from $T$ by replacing each bare path $P_i$ with the tree $\tilde P_i$. Note that $\tilde T$ has $n$ vertices, maximum degree at most $\Delta(T)\le \Delta$, and at least $k\ge \alpha n$ leaves, one from each of the $k$ bare paths we modified. Therefore, by Theorem~\ref{thm:manyleaves}, $G$ contains a copy of $\tilde T$, and thus $G^2$ contains a copy of $T$.
\end{proof}
\section{Concluding remarks}\label{section:conclusion}
In this paper, we solved a question of Krivelevich (Question~\ref{problem:krivelevich}) about whether the square of  $(n,d,\lambda)$-graphs contain spanning bounded degree trees. While doing so, we also made progress towards a question of Alon, Krivelevich, and Sudakov (Question~\ref{question:alon}), giving an affirmative answer provided the tree has linearly many leaves. Actually, in Theorem~\ref{thm:general} we can deal with trees with $\Theta(\frac{\lambda n}{d})$ leaves which, because of Lemma~\ref{lemma:paths-leaves}, implies that it only remains to solve the case when $T$ has a collection of $\Theta(\frac{\lambda n}{d})$ vertex-disjoint bare paths, each of length $\Theta(\frac{d}{\lambda})$. If $d/\lambda=\text{polylog}(n)$, then it seems plausible that one can use the absorption approach introduced by Montgomery~\cite{montgomery2019spanning} to deal with trees with $\Omega(n/\text{polylog}(n))$ bare paths, but, if $d/\lambda$ is much smaller than $\log n$, then this question seems to be out of reach at the moment. It would be quite interesting, though, to answer Question~\ref{question:alon} when $d$ is large in terms of $n$, as in the recent work of Glock, Munha Correia, and Sudakov~\cite{glock2023hamilton} for Hamilton cycles. 

\begin{question}Is it true that for any $\Delta\in\mathbb N$ and $\alpha>0$, there exists a positive constant $C$ such that if $G$ is an $(n,d,\lambda)$-graph, with $d/\lambda\ge C$ and $d\ge n^\alpha$, then $G$ contains all bounded degree spanning trees?
\end{question}
\section*{Acknowledgements}
We thank Michael Krivelevich and Richard Montgomery for helpful discussions around Question~\ref{question:alon} and the extendability method.

\bibliographystyle{abbrv}
\bibliography{spanningtrees}
\end{document}